\documentclass[11pt,reqno]{amsart}

\usepackage{amssymb}

\usepackage{amsfonts}

\usepackage{amsmath}

\usepackage[all]{xy}

\usepackage{color}

\newtheorem{theorem}{Theorem}[section]

\newtheorem{lemma}[theorem]{Lemma}

\newtheorem{proposition}[theorem]{Proposition}

\newtheorem{Definition}[theorem]{Definition}

\newtheorem{Example}[theorem]{Example}

\newtheorem{Remark}[theorem]{Remark}

\newenvironment{remark}{\begin{Remark}\begin{em}}{\end{em}\end{Remark}}

\newenvironment{definition}{\begin{Definition}\begin{em}}{\end{em}\end{Definition}}

\DeclareMathOperator{\tr}{tr}

\address{Sejong Kim \\ Department of Mathematics, Chungbuk National University, Cheongju 28644, Korea}
\email{skim@chungbuk.ac.kr}

\begin{document}

\author[Sejong Kim]{Sejong Kim}

\title[Parameterized Wasserstein mean with its properties]{Parameterized Wasserstein mean with its properties}

\date{}
\maketitle

\begin{abstract}
A new least squares mean of positive definite matrices for the divergence associated with the sandwiched quasi-relative entropy has been introduced. It generalizes the well-known Wasserstein mean for covariance matrices of Gaussian distributions with mean zero, so we call it the parameterized Wasserstein mean. We investigate in this article norm inequality of the parameterized Wasserstein mean, give its bounds with respect to the Loewner order, and show the extended version of Lie-Trotter-Kato formula for the parameterized Wasserstein mean. Finally we show the log-majorzation properties of the parameterized Wasserstein mean by using the Cartan mean.

\vspace{5mm}


\end{abstract}

\section{Introduction}

The Fr\'{e}chet mean (or barycenter) is a natural average arising from the least squares mean when the space has a metric structure. On the other hand, it is not easy to know whether the Fr\'{e}chet mean exists on a metric space. It has been known from \cite{St03}, in general, that the Fr\'{e}chet mean exists uniquely on the Hadamard space, which is the complete metric space satisfying the semi-parallelogram law. A typical and important example of the Hadamard space is the open convex cone $\mathbb{P}_{m}$ of $m \times m$ positive definite matrices equipped with the Riemannian trace metric $\delta(A, B) = \Vert \log A^{-1/2} B A^{-1/2} \Vert_{2}$. For an $n$-tuple $\mathbb{A} = (A_{1}, \dots, A_{n}) \in \mathbb{P}_{m}^{n}$ of positive definite matrices and a positive probability vector $\omega = (w_{1}, \dots, w_{n})$ the Fr\'{e}chet mean (also called the Cartan mean, Karcher mean)
\begin{displaymath}
G(\omega; \mathbb{A}) = \underset{X \in \mathbb{P}_{m}}{\arg \min} \sum_{j=1}^{n} w_{j} \delta^{2}(X, A_{j}),
\end{displaymath}
has been widely studied in theoretical and computational aspects: see \cite{Ka77, LL0, LL1, LK, LP, Ya}.

Especially the Wasserstein metric space of probability measures with barycenters has been recently important in a variety of research fields: see \cite{AGS, Vi08, Vi03} and their bibliographies. There are several interesting results about Wasserstein barycenters on the set $\mathcal{P}^{2}(\mathbb{R}^{n})$ of all probability measures on the Euclidean space $\mathbb{R}^{n}$ with finite second moment \cite{AC, ABCM, Gel}, including the fixed point approach to the Wasserstein mean of Gaussian distributions. For $\mu, \nu \in \mathcal{P}^{2}(\mathbb{R}^{n})$ the $L_{2}$-Wasserstein metric is defined as
\begin{displaymath}
W_{2}(\mu, \nu) := \left\{ \underset{\pi \in \Pi (\mu, \nu)}{\inf} \int_{\mathbb{R}^{n}} \Vert x-y \Vert^{2} d\pi(x,y) \right\}^{1/2},
\end{displaymath}
where $\Pi (\mu, \nu)$ denotes the set of all couplings on $\mathbb{R}^{n} \times \mathbb{R}^{n}$ with marginals $\mu$ and $\nu$. In particular, the $L_{2}$-Wasserstein distance for two Gaussian distributions $\mu$ and $\nu$ with mean $0$ and covariance matrices $A, B$ is formulated as
\begin{displaymath}
\frac{1}{\sqrt{2}} W_{2}(\mu, \nu) = \left[ \tr \left( \frac{A + B}{2} \right) - \tr (A^{1/2} B A^{1/2})^{1/2} \right]^{1/2},
\end{displaymath}
where we consider that $A$ and $B$ are $m \times m$ positive definite matrices. Note that this metric, denoted as $d(A, B)$ and called the Bures-Wasserstein distance, coincides with the Bures distance of density matrices in quantum information theory and is the matrix version of the Hellinger distance of probability vectors.

For given $n$-tuple $\mathbb{A} = (A_{1}, \dots, A_{n}) \in \mathbb{P}_{m}^{n}$ and a positive probability vector $\omega = (w_{1}, \dots, w_{n})$ the Wasserstein mean is the least squares mean for the Bures-Wasserstein distance:
\begin{displaymath}
\Omega(\omega; \mathbb{A}) = \underset{X \in \mathbb{P}_{m}}{\arg \min} \sum_{j=1}^{n} w_{j} d^{2}(X, A_{j}).
\end{displaymath}
It has been shown that such a minimizer exists uniquely by using non-smooth analysis, convex duality and the theory of optimal transport \cite{AC} and by using matrix analysis \cite{BJL18}. Moreover, lots of interesting properties for the Wasserstein mean of positive definite matrices have been established: an iteration approach to the Wasserstein mean using the optimal transport map \cite{ABCM}, a log-majorization property of the Wasserstein mean \cite{BJL}, and several inequalities (in terms of Loewner order and operator norm) and an extended version of Lie-Trotter-Kato formula for the Wasserstein mean \cite{HK19}.

In recent works the sandwiched quasi-relative entropy as a parameterized version of fidelity has been introduced in \cite{FL, WWY}:
\begin{displaymath}
F_{t} (A, B) = \tr \left( A^{\frac{1-t}{2t}} B A^{\frac{1-t}{2t}} \right)^{t}, \ t \in (0, \infty).
\end{displaymath}
Note that the usual fidelity is the case $t = 1/2$ and it is a variant of the relative R\'{e}nyi entropy. Furthermore, it has been shown in \cite{BJL19} that the sandwiched quasi-relative entropy $F_{t}$ is strictly concave and the following minimization problem
\begin{displaymath}
\underset{X \in \mathbb{P}_{m}}{\arg \min} \sum_{j=1}^{n} w_{j} \left[ \tr ((1-t) A_{j} + t X) - F_{t}(A_{j}, X) \right].
\end{displaymath}
has a unique solution by Brouwer's fixed point theorem. So it generalizes the Wasserstein mean for $t = 1/2$, and we call it the \emph{parameterized Wasserstein mean}. In this paper we investigate norm inequality of the parameterized Wasserstein mean, give bounds of the parameterized Wasserstein mean with respect to the Loewner order, and show that the parameterized Wasserstein mean satisfies the extended version of Lie-Trotter-Kato formula.  Finally, we show the log-majorzation property of parameterized Wasserstein mean by using the Cartan mean.

\section{Symmetric weighted geometric mean}

Let $\mathbb{H}_{m}$ be the real vector space of all $m \times m$ Hermitian matrices.
Let $\mathbb{P}_{m} \subset \mathbb{H}_{m}$ be the open convex cone of all $m \times m$ positive definite matrices. The general linear group $GL_{m}$ of all $m \times m$ invertible matrices acts on $\mathbb{P}_{m}$ via congruence transformations $\Gamma_{M}(X) = M X M^{*}$ for $M \in GL_{m}$ and $X \in \mathbb{P}_{m}$. For any $A, B \in \mathbb{H}_{m}$ we write $A \leq B$ if $B - A$ is positive semi-definite, and $A < B$ if $B - A$ is positive definite. This is indeed a partial order on $\mathbb{H}_{m}$, known as the Loewner order.

Let $\Delta_{n}$ be the simplex of positive probability vectors in $\mathbb{R}^{n}$ convexly spanned by the unit coordinate vectors. Let $\mathbb{A} = (A_{1}, \dots, A_{n}) \in \mathbb{P}_{m}^{n}$, $\omega = (w_{1}, \dots, w_{n}) \in \Delta_{n}$, $\sigma \in S^{n}$ a permutation on $n$-letters, and $M \in GL_{m}$. For convenience, we denote as
\begin{displaymath}
\begin{split}
\omega_{\sigma} & := (w_{\sigma(1)}, \dots, w_{\sigma(n)}) \in \Delta_{n} \\
\mathbb{A}_{\sigma} & := (A_{\sigma(1)}, \dots, A_{\sigma(n)}) \in \mathbb{P}_{m}^{n} \\
M \mathbb{A} M^{*} & := (M A_{1} M^{*}, \dots, M A_{n} M^{*}) \in \mathbb{P}_{m}^{n} \\
\mathbb{A}^{-1} & := (A_{1}^{-1}, \dots, A_{n}^{-1}) \in \mathbb{P}_{m}^{n}.
\end{split}
\end{displaymath}

\begin{definition}
We define a \emph{symmetric weighted geometric mean} of positive definite matrices to be a map $\mathfrak{M}: \Delta_{n} \times \mathbb{P}_{m}^{n} \to \mathbb{P}_{m}$ that satisfies the following properties:
For $\mathbb{A} = (A_{1}, \dots, A_{n}), \mathbb{B} = (B_{1}, \dots, B_{n}) \in \mathbb{P}_{m}^{n}$, $\omega = (w_{1}, \dots, w_{n}) \in \Delta_{n}$, $\sigma \in S^{n}$, $M \in GL_{m}$, and $\mathbf{a} = (a_{1}, \dots, a_{n}) \in \mathbb{R}_{++}^{n}$, where $\mathbb{R}_{++} := (0, \infty)$, these are
\begin{itemize}
\item[(P1)] (Consistency with scalars)
$\mathfrak{M}(\omega; \mathbb{A}) = A_{1}^{w_{1}} \cdots A_{n}^{w_{n}}$ if the $A_{i}$'s commute;
\item[(P2)] (Joint homogeneity)
$\mathfrak{M}(\omega; a_{1} A_{1}, \dots, a_{n} A_{n}) = a_{1}^{w_{1}} \cdots a_{n}^{w_{n}} \mathfrak{M}(\omega; \mathbb{A})$;
\item[(P3)] (Permutation invariance)
$\mathfrak{M}(\omega_{\sigma}; \mathbb{A}_{\sigma}) = \mathfrak{M}(\omega; \mathbb{A})$;
\item[(P4)] (Monotonicity)
If $B_{i} \leq A_{i}$ for all $1 \leq i \leq n$, then $\mathfrak{M}(\omega; \mathbb{B}) \leq \mathfrak{M}(\omega; \mathbb{A})$;
\item[(P5)] (Continuity)
The map $\mathfrak{M}(\omega; \cdot)$ is continuous;
\item[(P6)] (Congruence invariance)
$\mathfrak{M}(\omega; M \mathbb{A} M^{*}) = M \mathfrak{M}(\omega; \mathbb{A}) M^{*}$;
\item[(P7)] (Joint concavity)
$\mathfrak{M}(\omega; \lambda \mathbb{A} + (1-\lambda) \mathbb{B}) \geq \lambda \mathfrak{M}(\omega; \mathbb{A}) + (1-\lambda) \mathfrak{M}(\omega; \mathbb{B})$ for $0 \leq \lambda \leq 1$;
\item[(P8)] (Self-duality)
$\mathfrak{M}(\omega; \mathbb{A}^{-1})^{-1} = \mathfrak{M}(\omega; \mathbb{A})$;
\item[(P9)] (Determinantal identity)
$\displaystyle \det \mathfrak{M}(\omega; \mathbb{A}) = \prod_{i=1}^{n} (\det A_{i})^{w_{i}}$;
\item[(P10)] (Arithmetic-Geometric-Harmonic weighted mean inequalities)
$$ \displaystyle \mathcal{H}(\omega; \mathbb{A}) := \left( \sum_{i=1}^{n} w_{i} A_{i}^{-1} \right)^{-1} \leq \mathfrak{M}(\omega; \mathbb{A}) \leq \sum_{i=1}^{n} w_{i} A_{i} =: \mathcal{A}(\omega; \mathbb{A}). $$
\end{itemize}
A map $\mathfrak{M}$ satisfying (P1)-(P10) except (P3) is called a (asymmetric) \emph{weighted geometric mean}.
\end{definition}

Note that the two-variable weighted geometric mean
\begin{displaymath}
\mathfrak{M}(w_{1}, w_{2}; A, B) = A^{1/2} (A^{-1/2} B A^{-1/2})^{w_{2}} A^{1/2} =: A \#_{w_{2}} B
\end{displaymath}
is uniquely determined by (P1) and (P6), and also fulfils (P1)-(P10). Moreover, the two-variable weighted geometric mean $A \#_{w_{2}} B$ is the unique (up to parameterization) geodesic on the Hadamard space $\mathbb{P}_{m}$ with the Riemannian trace metric.

There are many different kinds of symmetric weighted geometric means on the open convex cone $\mathbb{P}_{m}$ including the Ando-Li-Mathias (ALM) mean \cite{ALM} and Bini-Meini-Poloni (BMP) mean \cite{BMP}. Among them a natural and canonical mean is the least squares mean, called the \emph{Cartan mean}, which is the unique minimizer of the weighted sum of squares of the Riemannian trace metric $\delta$:
\begin{equation} \label{E:Cartan}
G(\omega; A_{1}, \ldots, A_{n}) = \underset{X \in \mathbb{P}_{m}} {\arg\min} \sum_{i=1}^{n} w_{i} \delta^{2}(X, A_{i}).
\end{equation}
In \cite{LL0}, Lawson and Lim verified that the Cartan mean $G$ satisfies all the properties (P1)-(P10). Computing appropriate derivatives as in \cite{Bh} yields that the Cartan mean $G(\omega; \mathbb{A})$ coincides with the unique solution $X \in \mathbb{P}_{m}$ of the Karcher equation
\begin{equation} \label{E:Karcher}
\sum_{i=1}^{n} w_{i} \log (X^{-1/2} A_{i} X^{-1/2}) = O.
\end{equation}
Recently, Yamazaki \cite{Ya} has shown a unique characterization of the Cartan mean among other symmetric weighted geometric means, and its generalization to the probability measures with finite second moment for the Riemannian trace metric has been proved in \cite{KLL}.

\begin{theorem} \label{T:Yamazaki} \cite{KLL,Ya}
Let the map $\mathfrak{M}: \Delta_{n} \times \mathbb{P}_{m}^{n} \to \mathbb{P}_{m}$ be the symmetric weighted geometric mean satisfying
\begin{equation} \label{E:Yamazaki}
\sum_{j=1}^{n} w_{j} \log A_{j} \leq 0 \ \Longrightarrow \ \mathfrak{M}(\omega; \mathbb{A}) \leq I
\end{equation}
for any $\mathbb{A} = (A_{1}, \dots, A_{n}) \in \mathbb{P}_{m}^{n}$ and $\omega = (w_{1}, \dots, w_{n}) \in \Delta_{n}$. Then $\mathfrak{M} = G$. Furthermore, the Cartan mean $G$ satisfies the property \eqref{E:Yamazaki}.
\end{theorem}

\section{Parameterized Wasserstein means}

Let $\mathbb{A} = (A_{1}, \dots, A_{n}) \in \mathbb{P}_{m}^{n}$, and let $\omega = (w_{1}, \dots, w_{n}) \in \Delta_{n}$. For any $t \in (0,1)$ the following minimization problem
\begin{equation} \label{E:minimization}
\underset{X \in \mathbb{P}_{m}}{\arg \min} \sum_{j=1}^{n} w_{j} \left[ \tr ((1-t) A_{j} + t X) - \tr \left( A_{j}^{\frac{1-t}{2t}} X A_{j}^{\frac{1-t}{2t}} \right)^{t} \right]
\end{equation}
has been solved in \cite{BJL19}, so it gives us a new multivariate matrix mean. We recall its known results in this section, and investigate more interesting consequences in the later sections.

Note that the quantity $\displaystyle F_{t}(A_{j}, X) = \tr \left( A_{j}^{\frac{1-t}{2t}} X A_{j}^{\frac{1-t}{2t}} \right)^{t}$, called the \emph{sandwiched quasi-relative entropy}, is a parameterized version of fidelity since $F_{\frac{1}{2}}$ is the usual fidelity. Furthermore, the objective function $\displaystyle \varphi_{t} (X) = \sum_{j=1}^{n} w_{j} \left[ \tr ((1-t) A_{j} + t X) - \tr \left( A_{j}^{\frac{1-t}{2t}} X A_{j}^{\frac{1-t}{2t}} \right)^{t} \right]$ is strictly convex and its gradient is given by
\begin{displaymath}
\nabla \varphi_{t} (X) = t \left[ I - \sum_{j=1}^{n} w_{j} \left( A_{j}^{\frac{1-t}{t}} \#_{1-t} X^{-1} \right) \right].
\end{displaymath}
To prove the existence and uniqueness of the minimization problem \eqref{E:minimization}, it is enough to show that the equation $\nabla \varphi_{t} (X) = 0$ has a positive definite solution. Note that
\begin{equation} \label{E:equation}
\nabla \varphi_{t} (X) = 0 \ \Longleftrightarrow \ X = \sum_{j=1}^{n} w_{j} \left( X^{1/2} A_{j}^{\frac{1-t}{t}} X^{1/2} \right)^{t}.
\end{equation}
It has been shown in \cite{BJL19} that the map $H: \mathbb{P}_{m} \to \mathbb{P}_{m}$ defined by $\displaystyle H(X) = \sum_{j=1}^{n} w_{j} \left( X^{1/2} A_{j}^{\frac{1-t}{t}} X^{1/2} \right)^{t}$ is a self-map on the closed interval $[\alpha I, \beta I] := \{ X \in \mathbb{H}_{m}: \alpha I \leq X \leq \beta I \}$, where
\begin{center}
$\displaystyle \alpha := \underset{1 \leq i \leq n}{\min} \lambda_{\min} (A_{j})$ \ and \ $\displaystyle \beta := \underset{1 \leq i \leq n}{\max} \lambda_{\max} (A_{j})$.
\end{center}
We denote as $\lambda_{\min} (A)$ and $\lambda_{\max} (A)$ the smallest and largest eigenvalues of $A$, respectively. By Brouwer's fixed point theorem, the map $H$ has a fixed point. This yields the existence and uniqueness of the minimizer of \eqref{E:minimization}.

\begin{definition}
Let $\mathbb{A} = (A_{1}, \dots, A_{n}) \in \mathbb{P}_{m}^{n}$ and $\omega = (w_{1}, \dots, w_{n}) \in \Delta_{n}$. For $t \in (0,1)$, the \emph{parameterized Wasserstein mean} $\Omega_{t}(\omega; \mathbb{A})$ is defined as
\begin{displaymath}
\Omega_{t}(\omega; \mathbb{A}) = \underset{X \in \mathbb{P}_{m}}{\arg \min} \sum_{j=1}^{n} w_{j} \left[ \tr ((1-t) A_{j} + t X) - \tr \left( A_{j}^{\frac{1-t}{2t}} X A_{j}^{\frac{1-t}{2t}} \right)^{t} \right].
\end{displaymath}
\end{definition}

\begin{theorem} \label{T:Para-Wass-Eq}
The parameterized Wasserstein mean $\Omega_{t}(\omega; \mathbb{A})$ is the unique positive definite matrix $X \in \mathbb{P}_{m}$ satisfying that
\begin{equation} \label{E:Para-Wass-Eq}
\sum_{j=1}^{n} w_{j} \left( A_{j}^{\frac{1-t}{t}} \#_{1-t} X^{-1} \right) = I,
\end{equation}
equivalently,
\begin{displaymath}
X = \sum_{j=1}^{n} w_{j} \left( X^{1/2} A_{j}^{\frac{1-t}{t}} X^{1/2} \right)^{t}.
\end{displaymath}
\end{theorem}

For given $\mathbb{A} = (A_{1}, \dots, A_{n}) \in \mathbb{P}_{m}^{n}$ and $\omega = (w_{1}, \dots, w_{n}) \in \Delta_{n}$, we denote as
\begin{displaymath}
\begin{split}
\mathbb{A}^{p} & = (\underline{A_{1}, \dots, A_{n}}, \dots, \underline{A_{1}, \dots, A_{n}}) \in \mathbb{P}_{m}^{np}, \\
\omega^{p} & = \frac{1}{p} (\underline{w_{1}, \dots, w_{n}}, \dots, \underline{w_{1}, \dots, w_{n}}) \in \Delta_{np},
\end{split}
\end{displaymath}
where the number of blocks in the last expression is $p$.

The following are some properties of parameterized Wasserstein mean, compared with those of the Cartan mean.
\begin{theorem} \label{T:Para-Wass}
Properties of parameterized Wasserstein mean.
\begin{itemize}
\item[(1)] $($Consistency with scalars$)$ $\displaystyle \Omega_{t}(\omega; \mathbb{A}) = \left( \sum_{j=1}^{n} w_{j} A_{j}^{1-t} \right)^{\frac{1}{1-t}}$ if the $A_{j}$'s commute.

\item[(2)] $($Homogeneity$)$ $\displaystyle \Omega_{t}(\omega; \alpha \mathbb{A}) = \alpha \Omega_{t}(\omega; \mathbb{A})$ for any positive scalar $\alpha$.

\item[(3)] $($Permutation invariance$)$ $\displaystyle \Omega_{t}(\omega_{\sigma}; \mathbb{A}_{\sigma}) = \Omega_{t}(\omega; \mathbb{A})$ for any permutation $\sigma$ on $\{ 1, \dots, n \}$.

\item[(4)] $($Repetition invariance$)$ $\displaystyle \Omega_{t}(\omega^{p}; \mathbb{A}^{p}) = \Omega_{t}(\omega; \mathbb{A})$ for any $p \in \mathbb{N}$.

\item[(5)] $($Unitary congruence invariance$)$ $\displaystyle \Omega_{t}(\omega; U \mathbb{A} U^{*}) = U \Omega_{t}(\omega; \mathbb{A}) U^{*}$ for any unitary $U$.

\item[(6)] $($Determinantal inequality$)$ $\displaystyle \det \Omega_{t}(\omega; \mathbb{A}) \geq \prod_{j=1}^{n} (\det A_{j})^{w_{j}}$.
\end{itemize}
Moreover, $X = \Omega_{t}(\omega; A_{1}, \dots, A_{n-1}, X)$ if and only if $X = \Omega_{t}(\hat{\omega}; A_{1}, \dots, A_{n-1})$, where $\displaystyle \hat{\omega} = \frac{1}{1 - w_{n}} (w_{1}, \dots, w_{n-1}) \in \Delta_{n-1}$.
\end{theorem}

\begin{proof}
Most of items can be proved by Theorem \ref{T:Para-Wass-Eq}, so we prove some.
\begin{itemize}
\item[(1)] Assume that all $A_{j}$'s commute, so they are simultaneously diagonalizable. Set $\displaystyle X = \left( \sum_{j=1}^{n} w_{j} A_{j}^{1-t} \right)^{\frac{1}{1-t}}$. Then $X$ also commutes with all the $A_{j}$'s, and is a solution of the equation \eqref{E:Para-Wass-Eq}. By uniqueness of the positive definite solution for the equation \eqref{E:Para-Wass-Eq}, $X = \Omega_{t}(\omega; \mathbb{A})$.

\item[(6)] Let $X = \Omega_{t}(\omega; \mathbb{A})$. Then $\displaystyle X = \sum_{j=1}^{n} w_{j} \left( X^{1/2} A_{j}^{\frac{1-t}{t}} X^{1/2} \right)^{t}$. By the arithmetic-Cartan mean inequality,
\begin{displaymath}
X \geq G \left( \omega; \left( X^{1/2} A_{1}^{\frac{1-t}{t}} X^{1/2} \right)^{t}, \dots, \left( X^{1/2} A_{n}^{\frac{1-t}{t}} X^{1/2} \right)^{t} \right).
\end{displaymath}
Applying Corollary 7.7.4 (e) in \cite{HJ} and the determinantal identity of Cartan mean, we have
\begin{displaymath}
\det X \geq \prod_{j=1}^{n} \det \left( X^{1/2} A_{j}^{\frac{1-t}{t}} X^{1/2} \right)^{t w_{j}} = (\det X)^{t} \prod_{j=1}^{n} (\det A_{j})^{(1-t) w_{j}}.
\end{displaymath}
Solving for $\det X$, we obtain the desired inequality.
\end{itemize}
\end{proof}

\begin{remark}
Using the strict concavity of the map $f: \mathbb{P}_{m} \to \mathbb{R}, \ f(A) = \log \det A$ in Theorem 7.6.6 in \cite{HJ}, we can not prove only the determinantal inequality of the parameterized Wasserstein mean, but also obtain the condition that the determinantal equality holds. Indeed, taking the map $f$ on the equation \eqref{E:Para-Wass-Eq} yields
\begin{displaymath}
\begin{split}
0 & = \log \det \left[ \sum_{j=1}^{n} w_{j} \left( A_{j}^{\frac{1-t}{t}} \#_{1-t} X^{-1} \right) \right] \\
& \geq \sum_{j=1}^{n} w_{j} \log \det \left( A_{j}^{\frac{1-t}{t}} \#_{1-t} X^{-1} \right)
= (1-t) \sum_{j=1}^{n} w_{j} \log \det A_{j} - (1-t) \log \det X,
\end{split}
\end{displaymath}
which we get the inequality by solving for $\det X$. Moreover, the equality of Theorem \ref{T:Para-Wass} (6) holds if and only if $A_{i}^{\frac{1-t}{t}} \#_{1-t} X^{-1} = A_{j}^{\frac{1-t}{t}} \#_{1-t} X^{-1}$ for all $i$ and $j$. By the definition of two-variable weighted geometric mean it is equivalent to $A_{i} = A_{j}$ for all $i$ and $j$.
\end{remark}

\begin{lemma} \label{L:identity}
Let $\omega = (w_{1}, \dots, w_{n}) \in \Delta_{n}$ and $\mathbb{A} = (A_{1}, \dots, A_{n}) \in \mathbb{P}_{m}^{n}$ with $0 < \alpha I \leq A_{j} \leq \beta I$ for all $j$ and some positive scalars $\alpha, \beta$. Then $\alpha I \leq \Omega_{t}(\omega; \mathbb{A}) \leq \beta I$ for any $1/2 \leq t < 1$.
\end{lemma}

\begin{proof}
Assume that $0 < \alpha I \leq A_{j} \leq \beta I$ for all $j = 1, \dots, n$. Let $1/2 \leq t < 1$ and set $X = \Omega_{t}(\omega; \mathbb{A})$. Since the congruence transformation and the map $A \mapsto A^{r}$ for $r \in [0,1]$ preserve the Loewner order, we have $\alpha^{\frac{1-t}{t}} X \leq X^{1/2} A_{j}^{\frac{1-t}{t}} X^{1/2} \leq \beta^{\frac{1-t}{t}} X$, and $\alpha^{1-t} X^{t} \leq \left( X^{1/2} A_{j}^{\frac{1-t}{t}} X^{1/2} \right)^{t} \leq \beta^{1-t} X^{t}$. Then
\begin{displaymath}
\alpha^{1-t} X^{t} \leq \sum_{j=1}^{n} w_{j} \left( X^{1/2} A_{j}^{\frac{1-t}{t}} X^{1/2} \right)^{t} \leq \beta^{1-t} X^{t}.
\end{displaymath}
So $\alpha^{1-t} X^{t} \leq X \leq \beta^{1-t} X^{t}$ by Theorem \ref{T:Para-Wass-Eq}, and hence, $\alpha I \leq X \leq \beta I$.
\end{proof}

\section{Inequalities of parameterized Wasserstein means}

In the following we let $\mathbb{A} = (A_{1}, \dots, A_{n}) \in \mathbb{P}_{m}^{n}$ and $\omega = (w_{1}, \dots, w_{n}) \in \Delta_{n}$.

\begin{theorem}
For $t \in (0,1)$
\begin{displaymath}
\Vert \Omega_{t}(\omega; \mathbb{A}) \Vert \leq \left( \sum_{j=1}^{n} w_{j} \Vert A_{j} \Vert^{1-t} \right)^{\frac{1}{1-t}},
\end{displaymath}
where $\Vert \cdot \Vert$ denotes the operator norm.
\end{theorem}

\begin{proof}
Let $X = \Omega_{t}(\omega; \mathbb{A})$. Then by \eqref{E:equation}, by the triangle inequality for the operator norm, by the fact that $\Vert A^{t} \Vert = \Vert A \Vert^{t}$ for any $A \in \mathbb{P}_{m}$ and $t \geq 0$, and by the sub-multiplicativity for the operator norm in \cite[Section 5.6]{HJ}
\begin{displaymath}
\begin{split}
\Vert \Omega_{t}(\omega; \mathbb{A}) \Vert = \Vert X \Vert & = \left\| \sum_{j=1}^{n} w_{j} \left( X^{1/2} A_{j}^{\frac{1-t}{t}} X^{1/2} \right)^{t} \right\| \\
& \leq \sum_{j=1}^{n} w_{j} \left\Vert \left( X^{1/2} A_{j}^{\frac{1-t}{t}} X^{1/2} \right)^{t} \right\Vert
\leq \Vert X \Vert^{t} \left[ \sum_{j=1}^{n} w_{j} \Vert A_{j} \Vert^{1-t} \right].
\end{split}
\end{displaymath}
Hence, by simplification for $\Vert X \Vert$, we obtain the desired inequality.
\end{proof}

\begin{proposition} \label{P:arithmetic-Wass}
For $1/2 \leq t < 1$
\begin{displaymath}
\Omega_{t}(\omega; \mathbb{A})^{\frac{1-t}{t}} \leq \sum_{j=1}^{n} w_{j} A_{j}^{\frac{1-t}{t}}.
\end{displaymath}
\end{proposition}

\begin{proof}
Let $X = \Omega_{t}(\omega; \mathbb{A})$. Then $\displaystyle X = \sum_{j=1}^{n} w_{j} \left( X^{1/2} A_{j}^{\frac{1-t}{t}} X^{1/2} \right)^{t}$. Since the function $f(A) = A^{r}$ for $1 \leq r \leq 2$ is convex on $\mathbb{P}_{m}$ from \cite[Theorem 1.5.8]{Bh}, we have
\begin{displaymath}
\Omega_{t}(\omega; \mathbb{A})^{\frac{1}{t}} = X^{\frac{1}{t}} = \left[ \sum_{j=1}^{n} w_{j} \left( X^{1/2} A_{j}^{\frac{1-t}{t}} X^{1/2} \right)^{t} \right]^{\frac{1}{t}} \leq \sum_{j=1}^{n} w_{j} X^{1/2} A_{j}^{\frac{1-t}{t}} X^{1/2}.
\end{displaymath}
By the simple calculation we obtain the desired inequality.
\end{proof}

\begin{remark}
The arithmetic-Wasserstein mean inequality
\begin{displaymath}
\Omega_{1/2}(\omega; \mathbb{A}) \leq \sum_{j=1}^{n} w_{j} A_{j}
\end{displaymath}
has been already proved in \cite{BJL18}, and Proposition \ref{P:arithmetic-Wass} for $t = 1/2$ also yields the inequality.
\end{remark}

\begin{theorem} \label{T:Upper-Lower}
The parameterized Wasserstein mean has the following lower and upper bounds with respect to the Loewner order:
\begin{displaymath}
\frac{1}{1-t} I - \frac{t}{1-t} \sum_{j=1}^{n} w_{j} A_{j}^{\frac{t-1}{t}} \leq \Omega_{t}(\omega; \mathbb{A}) \leq \left[ \frac{1}{1-t} I - \frac{t}{1-t} \sum_{j=1}^{n} w_{j} A_{j}^{\frac{1-t}{t}} \right]^{-1},
\end{displaymath}
where the second inequality holds when $\displaystyle I - t \sum_{j=1}^{n} w_{j} A_{j}^{\frac{1-t}{t}}$ is invertible.
\end{theorem}

\begin{proof}
Let $X = \Omega_{t}(\omega; \mathbb{A})$. By the two-variable arithmetic-geometric-harmonic mean inequalities we have
\begin{displaymath}
\left[ t A_{j}^{-\frac{1-t}{t}} + (1-t) X \right]^{-1} \leq A_{j}^{\frac{1-t}{t}} \#_{1-t} X^{-1} \leq t A_{j}^{\frac{1-t}{t}} + (1-t) X^{-1}.
\end{displaymath}
Since the weighted sum is operator monotone,
\begin{displaymath}
\sum_{j=1}^{n} w_{j} \left[ t A_{j}^{-\frac{1-t}{t}} + (1-t) X \right]^{-1} \leq I = \sum_{j=1}^{n} w_{j} A_{j}^{\frac{1-t}{t}} \#_{1-t} X^{-1} \leq t \sum_{j=1}^{n} w_{j} A_{j}^{\frac{1-t}{t}} + (1-t) X^{-1}.
\end{displaymath}
Solving the second inequality for $X$, we obtain the upper bound for the parameterized Wasserstein mean. Taking inverse on both sides of the first inequality and applying the arithmetic-harmonic mean inequality, we have
\begin{displaymath}
t \sum_{j=1}^{n} w_{j} A_{j}^{-\frac{1-t}{t}} + (1-t) X \geq \left[ \sum_{j=1}^{n} w_{j} \left[ t A_{j}^{-\frac{1-t}{t}} + (1-t) X \right]^{-1} \right]^{-1} \geq I.
\end{displaymath}
Solving this for $X$, we obtain the lower bound for the parameterized Wasserstein mean.
\end{proof}

The Lie-Trotter-Kato product formula of two bounded operators is not fundamental only in various research areas such as Lie theory and operator algebra, but is also widely used for Gold-Thompson trace inequality and majorization problem. It has been extended in \cite{HK} to multi-variable cases in terms of the multi-variable operator mean, what we call the \emph{multivariate Lie-Trotter mean}. It has been proved that the multi-variable mean satisfying (P10) the arithmetic-geometric-harmonic mean inequalities is the multivariate Lie-Trotter mean. Even though the Wasserstein mean does not satisfy the Wasserstein-harmonic mean inequality, it has been proved by using another lower bound in \cite{HK19} that the Wasserstein mean is also the multivariate Lie-Trotter mean. 
As an application of Theorem \ref{T:Upper-Lower} we now show that the parameterized Wasserstein mean is the multivariate Lie-Trotter mean.
\begin{lemma} \label{L:suff}
For $\epsilon > 0$, let $\gamma : (-\epsilon, \epsilon) \to \mathbb{P}_{m}$ be a continuous map with $\gamma(0) = I$. Then for any $t \in (0,1)$ there exists a $\delta > 0$ such that $\gamma(s)^{\frac{1-t}{t}} > t I$ for all $s \in (-\delta, \delta)$.
\end{lemma}

\begin{proof}
Let $t \in (0,1)$. Since $\gamma : (-\epsilon, \epsilon) \to \mathbb{P}_{m}$ is a continuous map with $\gamma(0) = I$, there exists a $\delta > 0$ such that $\gamma(s) \in B_{r} (I) = \{ A \in \mathbb{H}_{m}: \Vert A - I \Vert < r \}$ for all $s \in (-\delta, \delta)$, where $r := 1 - t^{\frac{t}{1-t}} > 0$. That is,
\begin{displaymath}
| \lambda_{i}(\gamma(s)) - 1| = | \lambda_{i}(\gamma(s) - I) | \leq \Vert \gamma(s) - I \Vert < r,
\end{displaymath}
since $\gamma(s) - I \in \mathbb{H}_{m}$, where $\lambda_{i}(A)$ denotes the $i$th eigenvalue of $A \in \mathbb{H}_{m}$ in decreasing order. It implies that $\lambda_{i}(\gamma(s)) > 1-r = t^{\frac{t}{1-t}}$, so $\gamma(s) > t^{\frac{t}{1-t}} I$. Thus, $\gamma(s)^{\frac{1-t}{t}} > t I$.
\end{proof}

\begin{theorem} \label{T:Lie-Trotter}
The parameterized Wasserstein mean satisfies
\begin{displaymath}
\lim_{s \to 0} \Omega_{t}(\omega; \gamma_{1}(s), \dots, \gamma_{n}(s))^{1/s} = \exp \left( \sum_{j=1}^{n} w_{j} \gamma_{j}'(0) \right),
\end{displaymath}
where for $\epsilon > 0$, $\gamma_{j}: (- \epsilon, \epsilon) \to \mathbb{P}_{m}$ are differentiable curves with $\gamma_{j}(0) = I$ for all $j = 1, \dots, n$.
\end{theorem}

\begin{proof}
Let $\omega = (w_{1}, \dots, w_{n}) \in \Delta_{n}$ and let $\gamma_{1}, \dots, \gamma_{n} : (-\epsilon, \epsilon) \to \mathbb{P}_{m}$ be differentiable curves with $\gamma_{j}(0) = I$ for all $j$. By Lemma \ref{L:suff} there exists a sufficiently small $\delta > 0$ so that $\gamma_{j}(s)^{\frac{1-t}{t}} > t I$ for all $j$ and $s \in (-\delta, \delta)$. Then $\displaystyle \sum_{j=1}^{n} w_{j} \gamma_{j}(s)^{\frac{t-1}{t}} < \frac{1}{t} I$, and $\displaystyle I - t \sum_{j=1}^{n} w_{j} \gamma_{j}(s)^{\frac{t-1}{t}} > 0$ for any $s \in (-\delta, \delta)$.

By Theorem \ref{T:Upper-Lower}, we have
\begin{displaymath}
\frac{1}{1-t} I - \frac{t}{1-t} \sum_{j=1}^{n} w_{j} \gamma_{j}(s)^{\frac{t-1}{t}} \leq \Omega_{t}(\omega; \gamma_{1}(s), \dots, \gamma_{n}(s)) \leq \left[ \frac{1}{1-t} I - \frac{t}{1-t} \sum_{j=1}^{n} w_{j} \gamma_{j}(s)^{\frac{1-t}{t}} \right]^{-1}.
\end{displaymath}
Taking logarithms, using the operator monotonicity of the logarithm map, and multiplying all terms by $1/s$ for $s > 0$, we get
\begin{equation} \label{E:squeeze}
\begin{split}
\displaystyle \log \left[ \frac{1}{1-t} I - \frac{t}{1-t} \sum_{j=1}^{n} w_{j} \gamma_{j}(s)^{\frac{t-1}{t}} \right]^{1/s} & \leq \log \Omega(\omega; \gamma_{1}(s), \dots, \gamma_{n}(s))^{1/s} \\
& \leq \log \left[ \frac{1}{1-t} I - \frac{t}{1-t} \sum_{j=1}^{n} w_{j} \gamma_{j}(s)^{\frac{1-t}{t}} \right]^{-1/s}.
\end{split}
\end{equation}
Note that
\begin{displaymath}
\begin{split}
\displaystyle \lim_{s \to 0^{+}} \frac{1}{s} \log \left[ \frac{1}{1-t} I - \frac{t}{1-t} \sum_{j=1}^{n} w_{j} \gamma_{j}(s)^{\frac{t-1}{t}} \right] & = \sum_{j=1}^{n} w_{j} \gamma_{j}'(0), \\
\lim_{s \to 0^{+}} \frac{1}{s} \log \left[ \frac{1}{1-t} I - \frac{t}{1-t} \sum_{j=1}^{n} w_{j} \gamma_{j}(s)^{\frac{1-t}{t}} \right]^{-1} & = \sum_{j=1}^{n} w_{j} \gamma_{j}'(0).
\end{split}
\end{displaymath}
Taking the limit as $s \to 0^{+}$ in \eqref{E:squeeze}, we obtain
\begin{displaymath}
\displaystyle \lim_{s \to 0^{+}} \log \Omega(\omega; \gamma_{1}(s), \dots, \gamma_{n}(s))^{1/s} = \sum_{j=1}^{n} w_{j} \gamma_{j}'(0).
\end{displaymath}
Since the logarithm map $\log: \mathbb{P}_{m} \to \mathbb{H}_{m}$ is diffeomorphic, we get the desired identity. By the similar argument for $t < 0$, we obtain the conclusion.
\end{proof}

The notions of operator convexity and concavity are characterized by Jensen type inequalities in \cite{HP}. For every contraction $X$ we have
\begin{equation}
(X^{*} A X)^{r} \leq X^{*} A^{r} X \hspace{5mm} \textrm{if} \ \ 1 \leq r \leq 2,
\end{equation}
and
\begin{equation} \label{E:Hansen}
(X^{*} A X)^{r} \geq X^{*} A^{r} X \hspace{5mm} \textrm{if} \ \ 0 \leq r \leq 1.
\end{equation}
For $X \in GL_{m}$ such that its inverse $X^{-1}$ is a contraction,
\begin{equation} \label{E:Hansen-1}
(X^{*} A X)^{r} \leq X^{*} A^{r} X \hspace{5mm} \textrm{if} \ \ 0 \leq r \leq 1.
\end{equation}

\begin{theorem}
Let $t \in (0,1)$. Then
\begin{itemize}
\item[(1)] $\Omega_{t}(\omega; \mathbb{A}) \geq I$ implies $\displaystyle \mathcal{A}(\omega; A_{1}^{1-t}, \dots, A_{n}^{1-t}) \geq I$, and
\item[(2)] $\Omega_{t}(\omega; \mathbb{A}) \leq I$ implies $\Omega_{t}(\omega; \mathbb{A}) \leq \mathcal{H}(\omega; A_{1}^{t-1}, \dots, A_{n}^{t-1})$.
\end{itemize}
\end{theorem}

\begin{proof}
Let $t \in (0,1)$.
\begin{itemize}
\item[(1)] Assume that $X = \Omega_{t}(\omega; \mathbb{A}) \geq I$. Then $X^{-1} \leq I$, and by \eqref{E:Hansen-1}
\begin{displaymath}
\left( X^{1/2} A_{j}^{\frac{1-t}{t}} X^{1/2} \right)^{t} \leq X^{1/2} A_{j}^{1-t} X^{1/2}.
\end{displaymath}
Thus, by \eqref{E:equation} and the above inequality
\begin{displaymath}
I = \sum_{j=1}^{n} w_{j} X^{-1/2} \left( X^{1/2} A_{j}^{\frac{1-t}{t}} X^{1/2} \right)^{t} X^{-1/2} \leq \sum_{j=1}^{n} w_{j} A_{j}^{1-t}.
\end{displaymath}

\item[(2)] Assume that $X = \Omega_{t}(\omega; \mathbb{A}) \leq I$. Then by \eqref{E:equation} and \eqref{E:Hansen}
\begin{displaymath}
I \geq \sum_{j=1}^{n} w_{j} \left( X^{1/2} A_{j}^{\frac{1-t}{t}} X^{1/2} \right)^{t} \geq X^{1/2} \left( \sum_{j=1}^{n} w_{j} A_{j}^{1-t} \right) X^{1/2},
\end{displaymath}
so $\displaystyle X^{-1} \geq \sum_{j=1}^{n} w_{j} A_{j}^{1-t}$. Thus, we obtain (2) by taking inverse on both sides.
\end{itemize}
\end{proof}



For $1 \leq i, j \leq n$ let $A_{ij} \in M_{m}$, the set of all $m \times m$ matrices with entries in the field of complex numbers. We define a map $\Phi: M_{n}(M_{m}) \to M_{m}$ as
\begin{equation} \label{E:Phi}
\Phi \left(
\left[
  \begin{array}{ccc}
    A_{11} & \cdots & A_{1n} \\
    \vdots & \ddots & \vdots \\
    A_{n1} & \cdots & A_{nn} \\
  \end{array}
\right] \right) = \sum_{j=1}^{n} w_{j} A_{jj}.
\end{equation}
Then one can easily see that $\Phi$ is a positive linear and unital map.

\begin{theorem} \label{T:Kantorovich}
Let $\omega = (w_{1}, \dots, w_{n}) \in \Delta_{n}$. Let $\mathfrak{M}^{\omega} = \mathfrak{M}(\omega; \cdot): \mathbb{P}_{m}^{n} \to \mathbb{P}_{m}$ be the map satisfying the inequality
\begin{equation} \label{E:bound-harmonic}
\mathfrak{M}(\omega; A_{1}, \dots, A_{n}) \geq \mathcal{H}(\omega; A_{1}, \dots, A_{n}).
\end{equation}
If there exist positive scalars $\alpha$ and $\beta$ such that $0 < \alpha I \leq A_{j} \leq \beta I$ for all $j$, then
\begin{displaymath}
\mathfrak{M} \left( \omega; X^{-1} \#_{t} A_{1}^{\frac{1-t}{t}}, \dots, X^{-1} \#_{t} A_{n}^{\frac{1-t}{t}} \right) \geq \frac{4 \alpha \beta}{(\alpha + \beta)^{2}} I
\end{displaymath}
for any $1/2 \leq t < 1$, where $X = \Omega_{t}(\omega; A_{1}, \dots, A_{n})$.
\end{theorem}

\begin{proof}
For some positive scalars $\alpha$ and $\beta$ such that $0 < \alpha I \leq A_{j} \leq \beta I$ for all $j$, we have that $\alpha I \leq X = \Omega_{t}(\omega; A_{1}, \dots, A_{n}) \leq \beta I$ for $1/2 \leq t < 1$ by Lemma \ref{L:identity}, and $\alpha I \leq (X^{1/2} A_{j}^{\frac{1-t}{t}} X^{1/2})^{t} \leq \beta I$ for all $j$. So
\begin{displaymath}
\alpha I \leq
\left[
  \begin{array}{ccc}
    (X^{1/2} A_{1}^{\frac{1-t}{t}} X^{1/2})^{t} & \cdots & O \\
    \vdots & \ddots & \vdots \\
    O & \cdots & (X^{1/2} A_{n}^{\frac{1-t}{t}} X^{1/2})^{t} \\
  \end{array}
\right] \leq \beta I.
\end{displaymath}
Applying Proposition 2.7.8 in \cite{Bh} to the positive linear map $\Phi$, we obtain
\begin{displaymath}
\begin{split}
& \Phi \left(
\left[
  \begin{array}{ccc}
    (X^{1/2} A_{1}^{\frac{1-t}{t}} X^{1/2})^{t} & \cdots & O \\
    \vdots & \ddots & \vdots \\
    O & \cdots & (X^{1/2} A_{n}^{\frac{1-t}{t}} X^{1/2})^{t} \\
  \end{array}
\right] \right) \\
& \leq \frac{(\alpha + \beta)^{2}}{4 \alpha \beta}
\Phi \left(
\left[
  \begin{array}{ccc}
    (X^{1/2} A_{1}^{\frac{1-t}{t}} X^{1/2})^{t} & \cdots & O \\
    \vdots & \ddots & \vdots \\
    O & \cdots & (X^{1/2} A_{n}^{\frac{1-t}{t}} X^{1/2})^{t} \\
  \end{array}
\right]^{-1} \right)^{-1}.
\end{split}
\end{displaymath}
Equivalently, by Theorem \ref{T:Para-Wass-Eq}
\begin{displaymath}
X = \sum_{j=1}^{n} w_{j} (X^{1/2} A_{j}^{\frac{1-t}{t}} X^{1/2})^{t} \leq \frac{(\alpha + \beta)^{2}}{4 \alpha \beta} \left[ \sum_{j=1}^{n} w_{j} (X^{1/2} A_{j}^{\frac{1-t}{t}} X^{1/2})^{-t} \right]^{-1}.
\end{displaymath}
Taking the congruence transformation by $X^{-1/2}$ on both sides and applying the inequality \eqref{E:bound-harmonic}, we obtain
\begin{displaymath}
\begin{split}
I & \leq
\frac{(\alpha + \beta)^{2}}{4 \alpha \beta} X^{-1/2} \left[ \sum_{j=1}^{n} w_{j} (X^{-1/2} A_{j}^{-\frac{1-t}{t}} X^{-1/2})^{t} \right]^{-1} X^{-1/2} \\
& = \frac{(\alpha + \beta)^{2}}{4 \alpha \beta} \left[ \sum_{j=1}^{n} w_{j} X \#_{t} A_{j}^{-\frac{1-t}{t}} \right]^{-1} \\
& \leq \frac{(\alpha + \beta)^{2}}{4 \alpha \beta} \mathfrak{M} \left( \omega; X^{-1} \#_{t} A_{1}^{\frac{1-t}{t}}, \dots, X^{-1} \# A_{n}^{\frac{1-t}{t}} \right).
\end{split}
\end{displaymath}
\end{proof}

\section{Log-Majorization}

Let $\mathbf{x} = (x_{1}, \dots, x_{m})$ and $\mathbf{y} = (y_{1}, \dots, y_{m})$ be two $m$-tuples of nonnegative numbers. Let $x_{1}^{\downarrow} \geq x_{2}^{\downarrow} \geq \cdots \geq x_{m}^{\downarrow}$ be the decreasing rearrangement of $x_{1}, \dots, x_{m}$. If for all $1 \leq k \leq m$
\begin{displaymath}
\prod_{j=1}^{k} x_{j}^{\downarrow} \leq \prod_{j=1}^{k} y_{j}^{\downarrow},
\end{displaymath}
then we say that $\mathbf{x}$ is weakly log-majorized by $\mathbf{y}$, and write it as $\displaystyle \mathbf{x} \prec_{w \log} \mathbf{y}$. In addition, if the equality holds for $k = m$, then we say that $\mathbf{x}$ is log-majorized by $\mathbf{y}$, and write it as $\displaystyle \mathbf{x} \prec_{\log} \mathbf{y}$.

A standard technique in the theory of log majorization is the use of antisymmetric tensor powers. For $1 \leq k \leq m$ we denote by $Q_{k,m}$ the set of multi-indices $\alpha = (\alpha_{1}, \dots, \alpha_{k})$ with $1 \leq \alpha_{1} < \cdots < \alpha_{k} \leq m$. Let $A \in M_{m}$ and let $\alpha, \beta \in Q_{k,m}$. Then $A[\alpha | \beta]$ denotes the matrix obtained from $A$ by picking its entries from the rows corresponding to $\alpha$ and the columns corresponding to $\beta$. Recall that $\Lambda^{k}$ is a map assigning each $A \in M_{m}$ to an
$\left(
   \begin{array}{c}
     m \\
     k \\
   \end{array}
 \right)
 \times
 \left(
   \begin{array}{c}
     m \\
     k \\
   \end{array}
 \right)$
matrix $\Lambda^{k} A$ whose $(\alpha, \beta)$th entry for $\alpha, \beta \in Q_{k,m}$ is given by $\det A[\alpha | \beta]$, where the elements of $Q_{k,m}$ are ordered by the lexicographic ordering (or the dictionary order).
There are interesting properties for the antisymmetric tensor powers of positive matrix. Note that $\Lambda^{k} (c I) = c^{k} I$ for any constant $c$, where $I$ is the identity matrix with certain dimension, and
\begin{displaymath}
\prod_{j=1}^{k} \lambda_{j}^{\downarrow}(A) = \lambda_{1}^{\downarrow} (\Lambda^{k} A), \ \ 1 \leq k \leq m.
\end{displaymath}
The map $\mathbb{P}_{m} \ni A \mapsto \Lambda^{k} A$ is multiplicative, that is,
\begin{center}
$\Lambda^{k} (AB) = (\Lambda^{k} A) (\Lambda^{k} B)$ \ and \ $(\Lambda^{k} A)^{r} = \Lambda^{k} A^{r}, \ r \in (-\infty, \infty)$.
\end{center}
So it is clear that $\Lambda^{k} (A \#_{t} B) = (\Lambda^{k} A) \#_{t} (\Lambda^{k} B)$ for any $A, B \in \mathbb{P}_{m}$ and $t \in [0,1]$, and moreover, it can be extended to the symmetric weighted geometric means $\mathfrak{M}$ such as the ALM (Ando-Li-Mathias) mean, BMP (Bini-Meini-Poloni) mean, and Cartan mean $G$:
\begin{equation} \label{E:tensor power}
\Lambda^{k} \mathfrak{M}(\omega; A_{1}, \dots, A_{n}) = \mathfrak{M}(\omega; \Lambda^{k} A_{1}, \dots, \Lambda^{k} A_{n}).
\end{equation}

It has been shown in \cite{HK} that the map $\mathfrak{M}(\omega; \cdot): \mathbb{P}_{m}^{n} \to \mathbb{P}_{m}$ satisfying (P10) the arithmetic-geometric-harmonic weighted mean inequalities for given $\omega = (w_{1}, \dots, w_{n}) \in \Delta_{n}$ is the multivariate Lie-Trotter mean, as an extended version of the Lie-Trotter-Kato formula:
\begin{displaymath}
\lim_{s \to 0} \mathfrak{M}(\omega; \gamma_{1}(s), \dots, \gamma_{n}(s))^{1/s} = \exp \left( \sum_{j=1}^{n} w_{j} \gamma_{j}'(0) \right),
\end{displaymath}
where for $\epsilon > 0$, $\gamma_{j}: (-\epsilon, \epsilon) \to \mathbb{P}_{m}$ are any differentiable curves with $\gamma_{j}(0) = I$ for all $j$. In particular, taking $\gamma_{j}(s) = A_{j}^{s}$ for each $A_{j} \in \mathbb{P}_{m}$ we obtain
\begin{lemma} \label{L:Lie-Trotter-Kato}
Let the map $\mathfrak{M}(\omega; \cdot): \mathbb{P}_{m}^{n} \to \mathbb{P}_{m}$ satisfy (P10) the arithmetic-geometric-harmonic weighted mean inequalities. Then for given $\mathbb{A} = (A_{1}, \dots, A_{n}) \in \mathbb{P}_{m}^{n}$ and $\omega = (w_{1}, \dots, w_{n}) \in \Delta_{n}$
\begin{equation}
\lim_{s \to 0} \mathfrak{M}(\omega; A_{1}^{s}, \dots, A_{n}^{s})^{1/s} = \exp \left( \sum_{j=1}^{n} w_{j} \log A_{j} \right),
\end{equation}
where $\displaystyle L(\omega; \mathbb{A}) = \exp \left( \sum_{j=1}^{n} w_{j} \log A_{j} \right)$ is the log-Euclidean mean.
\end{lemma}


\begin{theorem}
Let $\mathbb{A} = (A_{1}, \dots, A_{n}) \in \mathbb{P}_{m}^{n}$ and $\omega = (w_{1}, \dots, w_{n}) \in \Delta_{n}$. For $t \in (0,1)$,
\begin{displaymath}
\lambda(G(\omega; \mathbb{A})) \prec_{\log} \lambda(L(\omega; \mathbb{A})) \prec_{w \log} \lambda(\Omega_{t}(\omega; \mathbb{A})).
\end{displaymath}
\end{theorem}

\begin{proof}
The first log-majorization $\lambda(G(\omega; \mathbb{A})) \prec_{\log} \lambda(L(\omega; \mathbb{A}))$ has been proved in \cite{BJL}.

Let $X = \Omega_{t}(\omega; \mathbb{A})$. Then $\displaystyle I = \sum_{j=1}^{n} w_{j} \left( A_{j}^{\frac{1-t}{t}} \#_{1-t} X^{-1} \right)$. Since the function $f(A) = A^{s}$ for $0 < s < 1$ is operator concave on $\mathbb{P}_{m}$ from \cite[Theorem 4.2.3]{Bh},
\begin{displaymath}
\sum_{j=1}^{n} w_{j} \left( A_{j}^{\frac{1-t}{t}} \#_{1-t} X^{-1} \right)^{s} \leq \left[ \sum_{j=1}^{n} w_{j} \left( A_{j}^{\frac{1-t}{t}} \#_{1-t} X^{-1} \right) \right]^{s} = I.
\end{displaymath}
For the symmetric weighted geometric mean $\mathfrak{M}$ satisfying the monotonicity, (P10) and \eqref{E:tensor power}, we have
\begin{displaymath}
\mathfrak{M} \left( \omega; \left( A_{1}^{\frac{1-t}{t}} \#_{1-t} X^{-1} \right)^{s}, \dots, \left( A_{n}^{\frac{1-t}{t}} \#_{1-t} X^{-1} \right)^{s} \right) \leq I,
\end{displaymath}
and moreover,
\begin{equation} \label{E:supplement}
\mathfrak{M} \left( \omega; \left( \Lambda^{k} A_{1}^{\frac{1-t}{t}} \#_{1-t} \Lambda^{k} X^{-1} \right)^{s}, \dots, \left( \Lambda^{k} A_{n}^{\frac{1-t}{t}} \#_{1-t} \Lambda^{k} X^{-1} \right)^{s} \right) \leq I.
\end{equation}

Assume that $\Lambda^{k} X \leq I$. Then $\Lambda^{k} X^{-1} \geq I$, so
\begin{displaymath}
\Lambda^{k} A_{j}^{\frac{1-t}{t}} \#_{1-t} \Lambda^{k} X^{-1} \geq \Lambda^{k} A_{j}^{\frac{1-t}{t}} \#_{1-t} I = \left( \Lambda^{k} A_{j}^{\frac{1-t}{t}} \right)^{t} = \Lambda^{k} A_{j}^{1-t}.
\end{displaymath}
By the Loewner-Heinz inequality, it implies that for $0 < s < 1$
\begin{displaymath}
\left( \Lambda^{k} A_{j}^{\frac{1-t}{t}} \#_{1-t} \Lambda^{k} X^{-1} \right)^{s}
\geq \left( \Lambda^{k} A_{j}^{1-t} \right)^{s} = \Lambda^{k} A_{j}^{(1-t) s}.
\end{displaymath}
Applying the monotonicity and \eqref{E:tensor power} of the mean $\mathfrak{M}$ to \eqref{E:supplement}, we have
\begin{displaymath}
\Lambda^{k} \mathfrak{M} (\omega; A_{1}^{(1-t) s}, \dots, A_{n}^{(1-t) s}) = \mathfrak{M} \left( \omega; \Lambda^{k} A_{1}^{(1-t) s}, \dots, \Lambda^{k} A_{n}^{(1-t) s} \right) \leq I.
\end{displaymath}
Taking $\frac{1}{(1-t) s}$ power on both sides yields
\begin{displaymath}
\Lambda^{k} \mathfrak{M} (\omega; A_{1}^{(1-t) s}, \dots, A_{n}^{(1-t) s})^{\frac{1}{(1-t) s}} \leq I.
\end{displaymath}
Letting $s \to 0$ and using Lemma \ref{L:Lie-Trotter-Kato}, we obtain that $\Lambda^{k} L(\omega; \mathbb{A}) \leq I$.

We have shown that for $1 \leq k < m$, $\Lambda^{k} \Omega_{t}(\omega; \mathbb{A}) \leq I$ implies $\Lambda^{k} L(\omega; \mathbb{A}) \leq I$. This yields that $\lambda_{1}^{\downarrow} (\Lambda^{k} L(\omega; \mathbb{A})) \leq \lambda_{1}^{\downarrow} (\Lambda^{k} \Omega_{t}(\omega; \mathbb{A}))$, that is,
\begin{displaymath}
\prod_{j=1}^{k} \lambda_{j}^{\downarrow}(L(\omega; \mathbb{A})) \leq \prod_{j=1}^{k} \lambda_{j}^{\downarrow}(\Omega_{t}(\omega; \mathbb{A})).
\end{displaymath}
From the determinantal inequality of parameterized Wasserstein mean in Theorem \ref{T:Para-Wass} (6), we can see that the above inequality still holds for $k = m$. Hence, the log-Euclidean mean $L(\omega; \mathbb{A})$ is weakly log-majorized by the parameterized Wasserstein mean $\Omega_{t}(\omega; \mathbb{A})$.
\end{proof}

The following shows the weak log-majorization between the Cartan mean of $p (\in (0,1))$ powers of given positive definite matrices and the $p$ power of parameterized Wasserstein mean of given positive definite matrices.
\begin{theorem}
Let $\mathbb{A} = (A_{1}, \dots, A_{n}) \in \mathbb{P}_{m}^{n}$ and $\omega = (w_{1}, \dots, w_{n}) \in \Delta_{n}$. For $t \in (0,1)$,
\begin{displaymath}
\lambda(G(\omega; A_{1}^{1-t}, \dots, A_{n}^{1-t})) \prec_{w \log} \lambda(\Omega_{t}(\omega; \mathbb{A}))^{1-t},
\end{displaymath}
where $\lambda(A)^{r} := (\lambda_{1}^{r}(A), \dots, \lambda_{m}^{r}(A))$ for any $A \in \mathbb{P}_{m}$ and $r \in \mathbb{R}$.
\end{theorem}

\begin{proof}
Let $X = \Omega_{t}(\omega; \mathbb{A})$. Then $\displaystyle I = \sum_{j=1}^{n} w_{j} \left( A_{j}^{\frac{1-t}{t}} \#_{1-t} X^{-1} \right)$. Since the logarithmic function $\log: \mathbb{P}_{m} \to \mathbb{H}_{m}$ is operator concave by Exercise 4.2.5 in \cite{Bh}, we have
\begin{displaymath}
0 = \log \left[ \sum_{j=1}^{n} w_{j} \left( A_{j}^{\frac{1-t}{t}} \#_{1-t} X^{-1} \right) \right] \geq \sum_{j=1}^{n} w_{j} \log \left( A_{j}^{\frac{1-t}{t}} \#_{1-t} X^{-1} \right).
\end{displaymath}
By Theorem \ref{T:Yamazaki} $G(\omega; A_{1}^{\frac{1-t}{t}} \#_{1-t} X^{-1}, \dots, A_{n}^{\frac{1-t}{t}} \#_{1-t} X^{-1}) \leq I$, and by the multiplicativity of antisymmetric tensor power and \eqref{E:tensor power}
\begin{equation} \label{E:majorization}
G \left( \omega; (\Lambda^{k} A_{1})^{\frac{1-t}{t}} \#_{1-t} (\Lambda^{k} X)^{-1}, \dots, (\Lambda^{k} A_{n})^{\frac{1-t}{t}} \#_{1-t} (\Lambda^{k} X)^{-1} \right) \leq I.
\end{equation}

Assume that $\Lambda^{k} X \leq I$ for $1 \leq k \leq m$. Taking the congruence transformation by $(\Lambda^{k} X)^{1/2}$ on both sides of \eqref{E:majorization} and applying \eqref{E:Hansen} yield
\begin{displaymath}
\begin{split}
\Lambda^{k} X & \geq G \left( \omega; ((\Lambda^{k} X)^{1/2} (\Lambda^{k} A_{1})^{\frac{1-t}{t}} (\Lambda^{k} X)^{1/2})^{t}, \dots, ((\Lambda^{k} X)^{1/2} (\Lambda^{k} A_{n})^{\frac{1-t}{t}} (\Lambda^{k} X)^{1/2})^{t} \right) \\
& \geq G \left( \omega; (\Lambda^{k} X)^{1/2} (\Lambda^{k} A_{1})^{1-t} (\Lambda^{k} X)^{1/2}, \dots, (\Lambda^{k} X)^{1/2} (\Lambda^{k} A_{n})^{1-t} (\Lambda^{k} X)^{1/2} \right).
\end{split}
\end{displaymath}
Taking the congruence transformation by $(\Lambda^{k} X)^{-1/2}$ on both sides implies
\begin{displaymath}
I \geq G \left( \omega; (\Lambda^{k} A_{1})^{1-t}, \dots, (\Lambda^{k} A_{n})^{1-t} \right) \\
= \Lambda^{k} G(\omega; A_{1}^{1-t}, \dots, A_{n}^{1-t}).
\end{displaymath}

We have shown that for $1 \leq k \leq m$, $\Lambda^{k} \Omega_{t}(\omega; \mathbb{A}) \leq I$ implies that $\Lambda^{k} G(\omega; A_{1}^{1-t}, \dots, A_{n}^{1-t}) \leq I$. Let $\alpha = \lambda_{1}^{\downarrow} (\Lambda^{k} \Omega_{t}(\omega; \mathbb{A}))^{1/k}$. Then by the homogeneity of parameterized Wasserstein mean in Theorem \ref{T:Para-Wass} (2)
\begin{displaymath}
\Lambda^{k} \Omega_{t} \left( \omega; \frac{1}{\alpha} \mathbb{A} \right) = \left( \Lambda^{k} \frac{1}{\alpha} I \right) \Lambda^{k} \Omega_{t} = \frac{1}{\alpha^{k}} \Lambda^{k} \Omega_{t} = \frac{1}{\lambda_{1}^{\downarrow} (\Lambda^{k} \Omega_{t})} \Lambda^{k} \Omega_{t} \leq I.
\end{displaymath}
It implies that
\begin{displaymath}
\begin{split}
I & \geq \Lambda^{k} G \left( \omega; \left( \frac{1}{\alpha} A_{1} \right)^{1-t}, \dots, \left( \frac{1}{\alpha} A_{n} \right)^{1-t} \right) \\
& = \left( \Lambda^{k} \frac{1}{\alpha^{1-t}} I \right) \Lambda^{k} G(\omega; A_{1}^{1-t}, \dots, A_{n}^{1-t}) \\
& = \frac{1}{\lambda_{1}^{\downarrow} (\Lambda^{k} \Omega_{t}(\omega; \mathbb{A}))^{1-t}} \Lambda^{k} G(\omega; A_{1}^{1-t}, \dots, A_{n}^{1-t}),
\end{split}
\end{displaymath}
that is, $\Lambda^{k} G(\omega; A_{1}^{1-t}, \dots, A_{n}^{1-t}) \leq \lambda_{1}^{\downarrow} (\Lambda^{k} \Omega_{t}(\omega; \mathbb{A}))^{1-t} I$. Thus,
\begin{displaymath}
\lambda_{1}^{\downarrow} \left( \Lambda^{k} G(\omega; A_{1}^{1-t}, \dots, A_{n}^{1-t}) \right) \leq \lambda_{1}^{\downarrow} \left( \Lambda^{k} \Omega_{t}(\omega; \mathbb{A}) \right)^{1-t}.
\end{displaymath}
By the determinantal inequality of parameterized Wasserstein mean in Theorem \ref{T:Para-Wass} (6), we obtain the weak log-majorization between $G(\omega; A_{1}^{1-t}, \dots, A_{n}^{1-t})$ and $\Omega_{t}(\omega; \mathbb{A})^{1-t}$.
\end{proof}

\vspace{4mm}

\textbf{Acknowledgement} \\

This work was supported by the National Research Foundation of Korea (NRF) grant funded by the Korea government (No. NRF-2018R1C1B6001394).

\end{document}